\documentclass{article}

\usepackage[utf8]{inputenc} 
\usepackage{url}

\usepackage{amsmath,amsfonts,amssymb}
\usepackage{amsthm}
\usepackage{graphicx}
\usepackage{enumerate}
\usepackage{multirow,bigdelim}
\usepackage{mathtools}
\usepackage{soul} 
\usepackage{pstricks}
\usepackage[margin=3cm]{geometry}
\usepackage{array}
\usepackage[normalem]{ulem}
\usepackage{url}
\usepackage{hyperref}
\usepackage{mathtools}
\hypersetup{
	colorlinks,
	linkcolor={red!50!black},
	citecolor={blue!50!black},
	urlcolor={blue!80!black}
}

\newcommand{\jProd}[2]{ {#1 \circ #2 } }	
\newcommand{\SOC}[1]{{\mathcal{Q}^{#1}}}
\newcommand{\CP}[1]{{\mathcal{CP}_{#1}}}
\newcommand{\COP}[1]{{\mathcal{COP}_{#1}}}

\newcommand{\interior}{\mathrm{int}\,}
\newcommand{\reInt}{\mathrm{ri}\,}
\newcommand{\reCone}{\mathrm{rec}\,}
\newcommand{\lineality}{\mathrm{lin}\,}

\newcommand{\bd}{\mathrm{bd}\,}

\newcommand{\conicHull}{\mathrm{cone}\,}

\newcommand{\norm}[1]{\lVert{#1}\rVert}
\newcommand{\inProd}[2]{\langle #1 , #2 \rangle }

\newcommand{\PSDcone}[1]{{\mathcal{S}^{#1}_+}}	
\newcommand{\nonNegative}[1]{{\mathcal{N}^{#1}}}	
\newcommand{\doubly}[1]{{\mathcal{D}^{#1}}}	
		
\newcommand{\eig}[2]{ \lambda_{#2}^{#1}}

\newcommand{\minFacePoint}[2]{ {\mathcal{F}(#1,#2)}}

\newcommand{\stdCone}{ {\mathcal{K}}}

\newcommand{\stdFace}{ \mathcal{F}}

\newcommand{\stdInt}{ {e}}

\newcommand{\matRank}{{\mathrm{ rank } \,}}

\renewcommand{\Re}{\mathbb{R}}    
  
\renewcommand{\S}{\mathcal{S}}                    
\newcommand{\tr}{\mathrm{trace}}

\newcommand{\dimSpace}{{\mathrm{dim}} \,}  

\newcommand{\jAlg}{\mathcal{E}}
\newcommand{\T}{\top\hspace{-1pt}} 
\newtheorem{definition}{Definition}
\newtheorem{lemma}[definition]{Lemma}
\newtheorem{proposition}[definition]{Proposition}

\newtheorem{example}[definition]{Example}

\newtheorem{corollary}[definition]{Corollary}

\newtheorem{theorem}[definition]{Theorem}
\newtheorem*{proposition*}{Proposition}
\theoremstyle{remark}

\title{A bound on the Carath\'{e}odory number}
\author{
	Masaru Ito%
		\thanks{Department of Mathematics, College of Science and Technology, Nihon University,
			1-8-14 Kanda-Surugadai, Chiyoda-Ku, Tokyo 101-8308, Japan 
			(\texttt{ito.m@math.cst.nihon-u.ac.jp)}.}
	\and
	Bruno F. Louren\c{c}o%
	\thanks{Department of Computer and Information Science, Faculty of Science and Technology, Seikei
		University, 3-3-1 Kichijojikitamachi, Musashino-shi, Tokyo 180-8633, Japan.
		(\texttt{lourenco@st.seikei.ac.jp)}}
}

\begin{document}
\maketitle
\begin{abstract}
The Carath\'{e}odory number $\kappa(\stdCone)$ of a pointed closed convex cone $\stdCone$ is the minimum among all the 
$\kappa$ for which every element of $\stdCone$ can be written as a nonnegative  linear combination 
of at most $\kappa$ elements belonging to extreme rays. Carath\'{e}odory's Theorem gives the bound 
 $\kappa(\stdCone) \leq \dim \stdCone$. 
 In this work we observe that this bound can be sharpened to
  $\kappa(\stdCone) \leq \ell_\stdCone - 1$, 
  where $\ell_\stdCone$ is the length of the longest chain of nonempty faces contained in $\stdCone$,
 thus tying the Carath\'{e}odory number with a key quantity that appears in 
 the analysis of facial reduction algorithms. We show that this bound is tight for several families of cones, which include symmetric cones and 
 the so-called smooth cones. We also give a family of examples showing that 
 this bound can also fail to be sharp. In addition, we furnish a new proof of a result by G{\"u}ler and Tun{\c{c}}el which states that the Carath\'{e}odory number of a 
 symmetric cone is equal to its rank. Finally, we connect our discussion to the notion of \emph{cp-rank} for completely 
 positive matrices.
\end{abstract}
\section{Introduction}
Let $\stdCone \subseteq \Re^n$ be a closed convex cone which is pointed, {\it i.e.}, $\stdCone \cap - \stdCone = \{ 0 \}$.
The \emph{Carath\'{e}odory number} of  $x \in \stdCone$ is the smallest nonnegative integer $\kappa (x)$ for which 
$$
x = d_1 + \ldots + d_{\kappa (x)},
$$ 
where each $d_i$ belongs to an extreme ray 
of $\stdCone$. We then define the  \emph{Carath\'{e}odory number of $\stdCone$} as
\begin{equation*}
\kappa (\stdCone) = \max \{ \kappa(x)\mid x \in \stdCone \}.
\end{equation*}
The Carath\'{e}odory number is a key geometric quantity and has a few
surprising connections. For instance, G\"uller and Tun\c{c}el showed that 
$\kappa(\stdCone)$ is a lower bound for the optimal barrier parameter for self-concordant 
barriers when $\stdCone$ is an homogeneous cone, see Proposition~4.1 in \cite{GO98}. 
When $\stdCone$ is, in fact, a symmetric cone, the inequality 
turns into an equality, see also the work by Tun\c{c}el and Truong \cite{TT03} and 
the related article by Tun\c{c}el and Xu \cite{TX01}.
Recently, in an article by Naldi \cite{Naldi2014}, $\kappa(\stdCone)$ was
 studied in the context of the so-called Hilbert cones.

The well-known Carath\'{e}odory Theorem tells us that the 
\emph{dimension of $\stdCone$} is an upper bound for $\kappa(\stdCone)$. In this note, we will show the  bound 
\begin{equation}\label{eq:main}
\kappa(\stdCone) \leq \ell_{\stdCone} - 1,
\end{equation}
where $\ell_{\stdCone}$ is the length of the longest chain of faces in $\stdCone$.  We 
remark that $\ell_{\stdCone}$ is an important quantity that appears in the analysis of 
\emph{facial reduction algorithms} (FRAs) \cite{Borwein1981495}. 
Namely, $\ell_{\stdCone} - 1$ is 
an upper bound for the minimum number of steps before a problem over $\stdCone$ is fully 
regularized. See \cite{article_waki_muramatsu, pataki_strong_2013} for a detailed discussion on 
facial reduction. 

Of course, $\ell_{\stdCone} - 1$ is itself bounded by $\dim \stdCone$, but here we will 
discuss several cases for which the former is strictly smaller than the latter, see Table~\ref{table:bounds_summary}.

The only extra assumption we will make is that $\stdCone$ must be a \emph{pointed cone}, that is, 
$\stdCone \cap -\stdCone = \{0\}$. This is done to ensure that $\stdCone$ has extreme rays.\footnote{Note 
that this is not a restrictive assumption. Letting $\lineality \stdCone = \stdCone \cap -\stdCone$,  
we have $
\stdCone = \stdCone \cap (\lineality \stdCone ^\perp) + \lineality \stdCone
$ and $\stdCone \cap (\lineality \stdCone ^\perp)$ is a pointed cone, for which 
our results apply.}

We now present a summary of the results. The bound 
\eqref{eq:main} is proven in Section~\ref{sec:main}. In fact, a slightly more general 
statement is proven, namely, that given $x \in \stdCone$, we have
$$\kappa(x) \leq \ell_\minFacePoint{x}{\stdCone} -1,$$
where $\minFacePoint{x}{\stdCone}$ is the minimal face of $\stdCone$ containing $x$. 
We will also show a family of cones for which the inequality in \eqref{eq:main} is strict.

Given $\stdCone$, there is a compact convex set $C$ such that $\stdCone$ 
is generated by $\{1\}\times C$. This process can also be reversed, so 
that given $C$, the cone generated by $\{1\}\times C$ is closed and pointed. 
Moreover, there is a correspondence between extreme points of $C$ and 
extreme rays of $\stdCone$. Therefore, the bound on
$\kappa (\stdCone)$ also induces a bound on $\kappa(C)$, namely
$$
\kappa (C) \leq \ell _C.
$$
This time, for $x \in C$, $\kappa (x)$ is the smallest integer for which 
we can write $x$ as a convex combination of $\kappa (x)$ extreme \emph{points}. 
As before, $\ell _C$ is the length of the longest chain of faces in $C$. This is discussed 
in Section~\ref{sec:main}.

The other contribution of this article is a discussion of several examples 
in which  \eqref{eq:main} turns into an equality. 
A part of Theorem~4.2 of \cite{TX01} shown by Tun\c{c}el and Xu asserts that 
a pointed polyhedral homogeneous convex $n$-dimensional cone $P$ satisfies 
$$
\kappa (P) = n.
$$
In this article, we will slightly generalize this fact so that we have
$$
\kappa (P) = \dim P = \ell _{P} - 1
$$
for any pointed polyhedral cone without the homogeneous hypothesis.
Moreover, we will strength the result and show that
whenever the set of extreme rays of $\stdCone$ is \emph{countable}, Equation \eqref{eq:main} turns 
into an equality. 

In \cite{LP15}, Liu and Pataki defined a \emph{smooth cone} as 
a pointed, full-dimensional cone $\stdCone$ for which all faces distinct 
from $\{0\}$ and $\stdCone$ are extreme rays. For those cones, 
we will show in Section~\ref{sec:tight} that \eqref{eq:main} holds with equality as well.

If $\stdCone$ is a symmetric cone, G\"uller and Tun\c{c}el showed 
in Lemma~4.1 of \cite{GO98} that
$$
\kappa (\stdCone) = \matRank \stdCone.
$$
However, the proof of Lemma 4.1 consists of a case-by-case analysis 
using the classification of Euclidean Jordan Algebras. There is also a proof in 
\cite{TT03} via the theory of homogeneous cones, see Theorem 8 therein.
We will give a new proof which we hope is simpler, using only elementary 
properties of Jordan Algebras. Moreover, we will also show 
that $\matRank \stdCone = \ell_{\stdCone} - 1$, which is a new result, as far as 
we know. These results are discussed in Section~\ref{sec:sym}.
We recall that among the symmetric cones we have the second order 
cone, the positive semidefinite cone and direct products of them.

Finally, in Section~\ref{sec:conc}, we discuss what is currently known 
about $\kappa(\stdCone)$ and $\ell _\stdCone$ for three families of 
cones that are of great interest recently: the copositive cone, the completely 
positive cone and the doubly nonnegative cone.
We also observe that the Carath\'{e}odory number of copositive matrices
coincides with the so-called \emph{cp-rank}.

\section{Preliminaries}\label{sec:prel}
Let $\stdCone$ be a closed convex cone contained in $\Re^n$. We denote its dual by 
$\stdCone ^* = \{x \in \Re^n \mid \inProd{x}{y} \geq 0, \forall y \in \stdCone \}$, 
where $\inProd{\cdot}{\cdot}$ is some inner product in $\Re^n$. Given a closed convex 
set $C$, we will denote  its dimension, interior, recession cone, relative interior and relative boundary by 
$\dimSpace C,  \interior{C}, \reCone C, \reInt C, \bd{C}$, respectively.
Recall that we have $\bd{C} = C \setminus \reInt C$.
 If $\stdFace \subseteq C$ is a convex set, we say that $\stdFace$ is a \emph{face} if 
the following condition holds: if $x,y \in C$ and $\alpha x + (1-\alpha)y \in \stdFace$ for 
some $0 < \alpha < 1$, then $x,y \in \stdFace$.

A face consisting of a single point is called an \emph{extreme point}. If $\stdCone$ is a pointed
closed convex cone, the only extreme point is zero. 
We refer to an one-dimensional face of $\stdCone$ as an \emph{extreme ray}. 
 
If $\stdFace _1$ and $\stdFace _2$ are  two faces of $C$, then 
$\reInt \stdFace _1 \cap \reInt \stdFace _2 \neq \emptyset$ if and only 
if $\stdFace _1 = \stdFace _2$, see Proposition~2.2 in \cite{pataki_handbook}. Now, suppose that $\stdFace _1 \subseteq \stdFace _2$. 
Then, a 
necessary and sufficient condition for the inclusion to be proper 
({\it i.e.},  $ \stdFace _1 \subsetneq \stdFace_2$ ) is that 
$\dim \stdFace _1 < \dim \stdFace _2$. See Corollaries~18.1.2 and 18.1.3 in \cite{rockafellar}.
The following result will also be useful.

\begin{lemma}[Theorem 6.4 in \cite{rockafellar}]\label{lemma:extend}
Let $\stdCone$ be a convex cone, $w \in \reInt \stdCone$ and $v \in \stdCone$.
Then, there is $\alpha > 1$ such that $\alpha w + (1-\alpha)v \in \stdCone$.
\end{lemma}

Let $d \in \stdCone$, then the generalized 
eigenvalue function of $\stdCone$ with respect to $d$ is 
\begin{equation*}
\eig{d}{\stdCone}(x) = \inf \{t \mid x - td \not \in \stdCone \}.
\end{equation*}
This function was introduced by Renegar in \cite{R15}.
We remark that when $\stdCone = \PSDcone{n}$ is the cone of positive semidefinite matrices and 
$d = I_n$ is the identity matrix, then $\eig{I_n}{\PSDcone{n}}$ is the usual  
minimum eigenvalue function. 
While in \cite{R15} the reference point $d$ is always a relative interior of 
$\stdCone$, an important twist here is that we will allow $d$ to be a relative boundary point of $\stdCone$.

We now collect a few properties of $\eig{d}{\stdCone}(x)$. 
\begin{lemma}\label{lemma:aux}
Let $d \in \stdCone$ with $d \neq 0$ and $x \in \Re^n$. The following assertions hold.
\begin{enumerate}[$(i)$]
\item If $x \in \stdCone $ then $\eig{d}{\stdCone}(x) < +\infty$.
\item If $x \in \reInt \stdCone$ then $ \eig{d}{\stdCone}(x) > 0$.
\item For $\alpha \in \Re$, we have $\eig{d}{\stdCone}(x+\alpha d) = \eig{d}{\stdCone}(x) + \alpha $.
\item If $x \in \stdCone $, then $x - \eig{d}{\stdCone}(x)d \in \stdCone \setminus {\reInt \stdCone}$.
\end{enumerate}
\end{lemma}
\begin{proof}
\begin{enumerate}[$(i)$]
\item Suppose that  $\eig{d}{\stdCone}(x) = +\infty$.
Then $x -td \in \stdCone$ for all $t \in \Re$, which implies that $d,-d \in \stdCone$. Since $\stdCone$ is pointed, we have $d = 0$, a contradiction.
\item If $x \in \reInt \stdCone$, by Lemma \ref{lemma:extend}, there exists $\alpha > 1$ such that 
$ \alpha x + (1-\alpha )d \in \stdCone$. This means that  $x -\frac{(\alpha-1)}{\alpha}d \in \stdCone$, 
which implies that $ \eig{d}{\stdCone}(x) \geq \frac{(\alpha-1)}{\alpha} > 0$.
\item It is clear from the definition of $\eig{d}{\stdCone}(x)$. Note that we use the convention that 
$+\infty + \alpha = +\infty$ and $-\infty + \alpha = -\infty$.
\item Due to items $(i)$ and $(iii)$, we have $\eig{d}{\stdCone}(x - \eig{d}{\stdCone}(x)d ) = 0$. By item $(ii)$, 
$x - \eig{d}{\stdCone}(x)d \not \in \reInt \stdCone$. Due to the definition of 
$\eig{d}{\stdCone}(x)$, for every $\epsilon > 0$, 
$$ 
x - (\eig{d}{\stdCone}(x)-\epsilon)d \in \stdCone,
$$ so that $x -\eig{d}{\stdCone}(x)d \in \stdCone$.
\end{enumerate}
\end{proof}

The next lemma is a classical result about the existence of extreme rays. We include  
the proof for the sake of self-containment.
\begin{lemma}\label{lemma:ext}
Let $\stdCone$ be a closed, pointed  convex cone with $\dim \stdCone \geq 1$.
Then $\stdCone$ contains at least one extreme ray, {\it i.e.}, an one dimensional face.
\end{lemma}
\begin{proof}
The first step is to pick $\stdInt^* \in \reInt \stdCone^*$ and 
show that
$$
C = \{x \in \stdCone \mid \inProd{x}{\stdInt^*} = 1 \}
$$ is compact. 
Since $C$ is a closed convex set, it suffices to 
show that its recession cone satisfies  $\reCone C = \{0\}$. We have 
$$
\reCone C = \{x \in \stdCone \mid \inProd{x}{\stdInt^*} = 0 \}.
$$ 
However, $x \in \reCone C$ if and only if $x \in (\stdCone^*)^\perp$, due to 
the choice of $\stdInt ^*$\footnote{This is a general fact. Suppose that $\stdCone$ is a convex cone,
$w \in \reInt \stdCone$ and $z \in \stdCone^*$. Then $z \in \stdCone^\perp$ if and only if $\inProd{w}{z} = 0$. To see that, suppose that $\inProd{w}{z} = 0$. Then, given $v \in \stdCone$ we have $\alpha w +(1-\alpha)v \in \stdCone$ for some $\alpha > 1$, 
by Lemma \ref{lemma:extend}. Taking the inner product with $z$, we see that $\inProd{w}{z}$ must be zero. }. As $(\stdCone^*)^\perp \subseteq \stdCone \cap - \stdCone $ (they are equal, in fact), we have $x = 0$.

Finally, we invoke the Krein-Milman Theorem, which implies that a nonempty compact convex set has at 
least one extreme point $z$. Then, one can verify that the half-line $h_z = \{\alpha z \mid \alpha \geq 0 \}$ 
is an extreme ray
of $\stdCone$.
\end{proof}

\section{Main result and discussion}\label{sec:main}
In what follows, we will denote by $\minFacePoint{S}{\stdCone}$ the minimal 
face of $\stdCone$ that contains $S \,(\subset \stdCone)$.
We also write $\minFacePoint{x}{\stdCone}$ when $S = \{x\}$.
Given a face $\stdFace$, we have $\stdFace = \minFacePoint{x}{\stdCone}$ 
if and only if $x \in \reInt \stdFace$, see Proposition~2.2 in \cite{pataki_handbook}.

A \emph{chain of faces of $\stdCone$} is a finite sequence of faces of $\stdCone$ 
such that each face properly contains the next. If we have a chain 
$\stdFace _1 \supsetneq \ldots \supsetneq \stdFace _\ell$, we define 
its \emph{length} as the  number of faces, which in this case is $\ell$. We will denote by $\ell _\stdCone$, the 
length of the longest chain of faces of $\stdCone$.

\begin{theorem}\label{theo:main}
Let $\stdCone$ be a pointed closed convex cone and $x \in \stdCone$. 
Then $$
\kappa(x) \leq \ell_\minFacePoint{x}{\stdCone} -1.$$
In particular, $\kappa (\stdCone) \leq \ell _\stdCone - 1$.
\end{theorem}
\begin{proof} 
Let $x \in \stdCone$. If $\dim \minFacePoint{x}{\stdCone} \leq 1$, we are done. Otherwise, 
due to Lemma~\ref{lemma:ext},  $\minFacePoint{x}{\stdCone}$ contains an extreme 
ray $\{\alpha d_1 \mid \alpha \geq 0 \}$. In particular, $d_1 \in \stdCone \setminus\{0\}$.
Now, let 
\begin{align*}
x_1 & = x - \eig{d_1}{\minFacePoint{x}{\stdCone}}(x)d_1\\
\stdFace _1 & = \minFacePoint{x_1}{\stdCone}. 
\end{align*}
Due to item $(iv)$ of Lemma~\ref{lemma:aux}, $x_1 \not \in \reInt \minFacePoint{x}{\stdCone}$, therefore 
$\stdFace _1 \subsetneq \minFacePoint{x}{\stdCone}$. We then proceed by induction, defining 
\begin{align*}
x_{i} & = x_{i-1} - \eig{d_i}{\stdFace _{i-1}}(x_{i-1})d_i,\\
\stdFace _i & = \minFacePoint{x_i}{\stdCone},
\end{align*}
where $\{\lambda d_i \mid \lambda \geq 0 \}$ is an extreme ray of $\stdFace _{i-1}$, 
which exists as long as $\dim \stdFace _{i-1} \geq 1$. Similarly, 
$x_i \not \in \reInt \stdFace _{i-1}$, so that $\stdFace _{i-1} \supsetneq \stdFace _i$.
Because we are in a finite dimensional space, there is an index $\ell$ for which 
 $\stdFace _\ell = \{0\}$, that is, $x_{\ell-1} - \eig{d_{\ell}}{\stdFace _{\ell - 1}}(x_{\ell-1})d_{\ell} = 0$.
 Unwinding the recursion, we can express $x$ as a positive linear combination of $\ell$ points belonging to extreme 
 rays and, at the same time, we obtain a chain of faces
 \begin{equation*}
 \minFacePoint{x}{\stdCone} \supsetneq \stdFace _1 \supsetneq \ldots \supsetneq \stdFace _\ell = \{0\}.
 \end{equation*}
 So that $\ell + 1 \leq \ell _\minFacePoint{x}{\stdCone}$. As any chain of faces of 
 $\minFacePoint{x}{\stdCone}$ is also a chain of $\stdCone$, we also obtain 
 $\ell + 1 \leq \ell _\stdCone$.
 
\end{proof}
Now, let $C$ be a nonempty compact convex set. Similarly, given $x \in C$, we may define the Carath\'{e}odory number of
$x$ as the minimum number $\kappa (x)$ necessary to express $x$ a convex combination of $\kappa (x)$ \emph{extreme points}.
Using Theorem~\ref{theo:main}, we can also say something about the Carath\'{e}odory number for compact convex sets thanks to 
the following well-known result. We include a proof in Appendix \ref{app:prop}.

\begin{proposition}\label{prop:cone_base}
Let $C \subseteq \Re^n$ be a nonempty compact convex set. Let 
$$
\stdCone  = \{(\alpha, \alpha x) \mid \alpha \geq 0, x \in C \}.	
$$
Then
\begin{enumerate}[$(i)$]
	\item $\stdCone $ is a pointed closed convex cone.
	\item Let $\stdFace$ be a face of $\stdCone $ that is not $\{0\}$, then 
	$$
	\stdFace _C = \{ x \in C \mid  (1,x) \in \stdFace \}
	$$is a face of $C$. Moreover, $\dim \stdFace _C = \dim \stdFace  - 1$. 
	\item Let $\stdFace _C$ be a face of $C$, then 
	$$
	\stdFace = \{ (\alpha, \alpha x)  \mid \alpha \geq 0, x \in \stdFace _C \}
	$$is a face of $\stdCone$. Moreover, $\dim \stdFace = \dim \stdFace _C + 1$.
	
\end{enumerate}
\end{proposition}

In a similar fashion, we will define $\ell _C$ as the 
length of the longest chain of faces in $C$.

\begin{theorem}
Let $C$ be a nonempty compact convex set and $x \in C$.
Then,
$$
\kappa (x) \leq \ell _{C}.
$$
\end{theorem}
\begin{proof}
Let $\stdCone$ be as in Proposition~\ref{prop:cone_base}, then the first step is showing that
$$
\kappa (x) = \kappa ( (1,x) ),
$$
where it is understood that $\kappa (1,x)$ is computed with respect to $\stdCone$. Suppose that 
\begin{equation}\label{eq:aux1}
(1,x) = (\alpha _1, \alpha _1 x_1) + \ldots + (\alpha _\ell, \alpha _\ell x_\ell)
\end{equation}
where each $(\alpha _i, \alpha _i x_i)$ lies in an extreme ray of $\stdCone$ and the $\alpha _i$ are positive.
Due to item $(ii)$ of Proposition~\ref{prop:cone_base}, it must be the case that the $x_i$ are 
extreme points of $C$. So Equation \eqref{eq:aux1} also expresses $x$ as a convex combination of 
$\ell$ extreme points. 
Conversely, if $x$ is expressed as a convex combination of $\ell$ extreme points, it is  also possible to express $(1,x)$ as a sum of $\ell$ extreme rays as in Equation \eqref{eq:aux1}.

If we show that $\ell _\stdCone = \ell _C + 1$, then the result will follow by Theorem~\ref{theo:main}.
To show that this is indeed the case, consider an arbitrary chain of faces of $C$
$$
\stdFace^1 _C \supsetneq \ldots \supsetneq \stdFace^\ell _C. 
$$
Then, following Proposition~\ref{prop:cone_base}, we also obtain a chain of faces of $\stdCone$ and we 
can enlarge the chain by adding the zero face.
$$
\stdFace^1 \supsetneq \ldots \supsetneq \stdFace^\ell  \supsetneq \{0\}. 
$$
This process can be reversed and any chain of faces of $\stdCone$ that does not contain the 
zero face also gives rise to a chain of faces in 
$C$.
\end{proof}

\section{Tightness of the bound}\label{sec:tight}
The main result of this section gives a few conditions ensuring that 
$\kappa (\stdCone) = \ell _{\stdCone} - 1$. We will also furnish an example 
where the bound fails to be tight. 

A pointed closed convex cone $\stdCone$ is said to be \emph{strictly convex}  if $\interior{ \stdCone } \ne \emptyset$
and we have $\stdFace(\{x, y\}, \stdCone) = \stdCone$ for every linearly independent $x,y \in \bd{ \stdCone }$, see \cite{barker73} and Definition 2.A.4 in \cite{barker81}. An equivalent concept is the notion of \emph{smooth cones}, which was considered in \cite{LP15}: $\stdCone$ is a smooth cone if $\interior{ \stdCone } \ne \emptyset$
and every face of $\stdCone$ different from $\{0\}$ and $\stdCone$ is an extreme ray.\footnote{
	In fact, if $\stdCone$ is a strictly convex cone, then any face $\stdFace$ with $\{0\} \subsetneq \stdFace \subsetneq \stdCone$ is one dimensional because every $x, y \in \stdFace $ cannot be linearly independent, 
	since $\stdFace \subset \bd{\stdCone}$. If $\stdCone$ is a smooth cone, on the other hand, every linearly independent $x, y \in \bd{ \stdCone }$ must satisfy $\stdFace(\{x, y\}, \stdCone) = \stdCone$ since we have $\dim{ \stdFace(\{x, y\}, \stdCone) } \geq 2$.
}
It is known that every strictly convex cone $\stdCone$ with $\dim\stdCone \geq 2$ satisfies $\kappa(\stdCone) = 2$, see Lemma~4.1 in  \cite{TX01}. In what follows, we give a new proof 
and we will point out the connection  to $\ell_ \stdCone$.

\begin{theorem}\label{theo:bounds}
Let $\stdCone$ be a pointed closed convex cone.
\begin{enumerate}[$(i)$]

\item If the set of extreme rays of $\stdCone$ is countable, then we have
\begin{equation}\label{eq:tight1}
\kappa (\stdCone) = \ell _{\stdCone} - 1 = \dim \stdCone.
\end{equation}
In particular, \eqref{eq:tight1} holds when $\stdCone$ is polyhedral.

\item If $\stdCone$ is a strictly convex cone with $\dim\stdCone \geq 2$, then we have
$$
\kappa (\stdCone) = \ell _{\stdCone} - 1 = 2.
$$
\end{enumerate}
\end{theorem}
\begin{proof}
	\begin{enumerate}[$(i)$]

		\item
		In view of the definition of $\kappa(\stdCone)$, we can write
		$$
		\stdCone =
			\bigcup
			\{
			\conicHull( \{d_1,\ldots,d_{\kappa(\stdCone)}\} )
			\mid
			d_i \text{ belongs to an extreme ray of } \stdCone,~ \norm{d_i} = 1
			\}
		$$
		where $\conicHull( \{d_1,\ldots,d_{\kappa(\stdCone)}\} )$ denotes the  smallest convex cone containing $\{d_1,\ldots,d_{\kappa(\stdCone)}\} $. Since the set of extreme rays 
		is countable, it follows that $\stdCone$ is a countable union of cones of dimension at most $\kappa(\stdCone)$. 
		
		This forces that $\kappa(\stdCone) \geq \dim{\stdCone}$ since a finite dimensional convex cone
		cannot be covered by a countable union of convex subsets with strictly smaller dimmension\footnote{
			Consider the linear span of $\stdCone$ endowed with the Lebesgue measure $\mu$, which 
			is a vector space of dimension $\dim \stdCone$. Also, recall that convex 
			sets are Lebesgue measurable. As $\stdCone$ contains
			balls of dimension $\dim \stdCone$, we have $\mu(\stdCone) > 0$.
			If $\stdCone$ is covered by a countable collection $\{V_i\}$ of convex sets 
			with dimension strictly smaller than $\dim \stdCone$, we arrive at an contradiction
			$\mu(\stdCone) \leq \sum_{i}\mu(V_i) = 0$, due to the countable subadditivity of $\mu$.
			See  \cite{khare09}, for an algebraic discussion in a more general context.
		}.
		Therefore, Theorem~\ref{theo:main} concludes that $\kappa(\stdCone) = \ell_\stdCone - 1 = \dim\stdCone$.
		
		As an immediate consequence, polyhedral cones satisfy \eqref{eq:tight1} because the set of extreme rays of a polyhedral cone is finite (see, {\it e.g.}, \cite[Theorem 19.1]{rockafellar}).
		
		\item
		The definition of smooth cones implies $\ell_\stdCone \leq 3$ since every chain of faces of $\stdCone$ cannot be longer than the one of the form $\stdCone \supsetneq \stdFace \supsetneq \{0\}$ where $\stdFace$ is an extreme ray.
		Furthermore, it is clear that $\kappa(\stdCone) > 1$ due to $\dim\stdCone \geq 2$. Hence, Theorem~\ref{theo:main} concludes that $\kappa(\stdCone) = \ell_\stdCone - 1 = 2$.
	\end{enumerate}
\end{proof}


We now give a family of examples where $\kappa (\stdCone) < \ell _{\stdCone} - 1$.
\begin{example}
Consider the cone
$$
\stdCone = \left\{ z = (x, t) \in \Re^{n} \times \Re ~\middle|~ t \geq \sqrt{x_1^2 + \cdots + x_n^2}, x_2 \geq 0, \ldots, x_n \geq 0 \right\} \subset \Re^{n+1}.
$$
(Notice that the component $x_1$ is allowed to be negative.)
Then, for $n\geq 1$ we have
$$
\kappa (\stdCone) = 2, \quad \ell _{\stdCone} = n + 2.
$$
\end{example}
\begin{proof}
We remark that  $\stdCone$ is the intersection of the second order cone $\SOC{n+1} = \{(x, t) \mid t \geq~ \sqrt{x_1^2 + \cdots x_n^2}\}$ and
the cone $\stdCone' := \{ (x, t) \mid x_2 \geq 0, \ldots, x_n \geq 0 \}$. 

Let us show that $\kappa (z) \leq 2$ for any $z \in \stdCone$. For $z = (x, t) \in \stdCone$, define
$$
\lambda_{\pm} := \pm \left( - x_1 + \sqrt{ t^2 - (x_2^2 + \cdots + x_n^2) } \right)
$$
so that
\begin{equation}
\sqrt{ (x_1 \pm \lambda_{\pm})^2 + x_2^2 + \cdots + x_n^2 } = t.\label{eq:ex}
\end{equation}
Now, let $e_1$ denote the unit vector along the first coordinate.
Equation \eqref{eq:ex} implies that the points $d_{\pm} := z \pm \lambda_{\pm} e_1 $ belong to both 
$\stdCone$ and the boundary of $\SOC{n+1}$. Since every boundary point of $\SOC{n+1}$ lies in an extreme ray of 
$\SOC{n+1}$ (see Example 2.6 in \cite{pataki_handbook}), we conclude that the points $d_{\pm}$ must belong to extreme rays of $\stdCone$ as 
well.
Since $z$ lies on the segment between $d_-$ and $d_+$, we have $\kappa (z) \leq 2$. As $\stdCone$ 
is not a single extreme ray, we must have $\kappa (\stdCone) = 2$.

Finally, we verify $\ell_{\stdCone} = n + 2$ as follows.
Define the faces $\{\stdFace_i\}_{i=1}^{n+2}$ of $\stdCone$ by $\stdFace_1 := \stdCone$,
\begin{eqnarray*}
	\stdFace_{i+1} &:=& \{ (x, t) \in \stdCone' \mid x_2 = \cdots = x_{i+1} = 0\} \cap \SOC{n+1}, \quad i = 1, \ldots, n-1, \\
	\stdFace_{n+1} &:=& \{ (x, t) \in \stdCone' \mid x_2 = \cdots = x_n = 0\} \cap \{(x,t) \in \SOC{n+1} \mid x_1 = t \geq 0 \},
\end{eqnarray*}
and $\stdFace_{n+2} := \{0\}$, which gives a chain of faces of $\stdCone$ of length $n+2$. Note 
that these are indeed faces of $\stdCone$, since they arise as intersections of faces of $\stdCone'$ 
and $\SOC{n+1}$.
Since $\stdCone$ is contained in a space of dimension $n+1$, it  must be indeed the largest possible 
chain.
\end{proof}

\section{The symmetric cone case}\label{sec:sym}

We say that $\stdCone$ is a \emph{symmetric cone} if $\stdCone = \stdCone^*$, $\interior \stdCone \neq \emptyset$ and for every pair of elements 
$x,y$ in the interior of $\stdCone$ there is an invertible linear transformation $T$ such that 
$T(\stdCone) = \stdCone$ and $T(x) = T(y)$. The theory of symmetric cones is strongly connected with the study of 
\emph{Euclidean Jordan Algebras}. The default reference is the book by Faraut and Kor\'{a}nyi \cite{FK94} but 
there are many introductory accounts  in the context of optimization, see \cite{Sturm2000, Faybusovich2008}. 
In this section, we will furnish in Theorem~\ref{theo:sym_car} another proof that $\kappa (x) = \matRank x$. 
Moreover in Theorem~\ref{theo:sym_long}, we will show 
that $\ell _{\stdCone} = \matRank \stdCone + 1$, which is a new result as far as we know.
The reader who is already familiar with the theory of Jordan Algebras can skip to Section~\ref{sec:car_sym}.

\subsection{Preliminaries}
Let $\jAlg$ be a finite dimensional real vector space equipped with a bilinear form $\jProd{}{} :\jAlg\times \jAlg \to \jAlg$. 
Write $x^2$ for $\jProd{x}{x}$. Now, suppose that $\jProd{}{}$ satisfies the following properties for all $x,y \in \jAlg$:
\begin{enumerate}
	\item $\jProd{x}{y} = \jProd{y}{x}$,
	\item $\jProd{x}({\jProd{x^2}{y}}) = \jProd{x^2}({\jProd{x}{y}})$.
\end{enumerate}
Then, $(\jAlg,\jProd{}{})$ is said to be a \emph{Jordan Algebra} and 
$\jProd{}{}$ is said to be a \emph{Jordan product}.

Furthermore, suppose that $\jAlg$ is equipped with an inner product $\inProd{\cdot}{\cdot}$ such that 
for all $x,y,z \in \jAlg$ we have
$$
\inProd{\jProd{x}{y}}{z} = \inProd{y}{\jProd{x}{z}}.
$$
Then, $(\jAlg,\jProd{}{})$ is said to be an \emph{Euclidean Jordan Algebra}.
Given a Jordan algebra, we define its cone of squares as
$$
\stdCone = \{ \jProd{x}{x} \mid x \in \jAlg \}.
$$
\begin{example}
Let $\S^n$ denote the space of $n\times n$ symmetric matrices equipped with the inner product such 
that $\inProd{A}{B} = \tr(AB)$, for $A,B \in \S^n$. Then, $\S^n$ is an Euclidean Jordan 
Algebra with the following product
$$
\jProd{A}{B} = \frac{AB+BA}{2}.
$$
The corresponding cone of squares is the cone of positive semidefinite matrices $\PSDcone{n}$.

Now, consider $\Re^{n+1}$ equipped with the usual Euclidean inner product. Given 
$x \in \Re^n$, write $x = (x_0,\overline{x})$, where $x_0 \in \Re$ and $\overline{x} \in \Re^n$.
Consider the following Jordan Products in $\Re^{n+1}$
\begin{align*}
\jProd{x}{y} = (\inProd{x}{y}, x_0\overline{y} + y_0 \overline{x}).
\end{align*}
The corresponding cone of squares is the second order cone $\SOC{n+1} = \{(x_0,\overline{x}) \mid x_0 \geq \sqrt{\inProd{\overline{x}}{\overline{x}}} \}$.

\end{example}

It is known that $\stdCone$ is a symmetric cone if and only if it arises as the cone of squares induced by 
some Euclidean Jordan Algebra $(\jAlg,\jProd{}{})$, see Theorems~III.2.1 and III.3.1 in \cite{FK94}. 
Moreover, the construction in Theorem III.3.1 shows that we can safely assume that $(\jAlg, \jProd{}{})$ has an identity element $e$. That is, $\jProd{x}{e} = x$ for all $x \in \jAlg$.

One of the beautiful aspects of the theory of Jordan Algebras is that it allows the definition of objects such as eigenvalues,
determinant, trace and rank in a very general context. In particular, we also 
have a version of the Spectral Theorem for Jordan Algebras. In what follows, we will say that $c$ is \emph{idempotent} if 
$\jProd{c}{c} = c$. Morover, $c$ is \emph{primitive} if it is nonzero and there is no way of writing 
$$
c = a+b,
$$
with $a$ and $b$ nonzero idempotent elements satisfying $\jProd{a}{b} = 0$.

\begin{theorem} [Spectral theorem, see Theorem III.1.2 in \cite{FK94}]\label{theo:spec}
	Let $(\jAlg, \jProd{}{} )$ be an Euclidean Jordan Algebra and let $x \in \jAlg$. Then there are:
	\begin{enumerate}
		\item primitive idempotents $c_1, \dots, c_r$ satisfying
		\begin{align}
		\jProd{c_i}{c_j} & = 0, \qquad \text{ for } i \neq j \\
		\jProd{c_i}{c_i} &= c_i, \qquad  i = 1, \ldots, r \\
		c_1 + \ldots + c_r & = e, \qquad i = 1, \ldots, r
		\end{align}
		\item unique real numbers $\lambda _1, \ldots, \lambda _r$ satisfying
		\begin{equation}		
		x = \sum _{i=1}^r \lambda _i c_i \label{eq:dec}.
		\end{equation}
	\end{enumerate}
\end{theorem}
We say that the $c_1, \ldots , c_r$ in Theorem~\ref{theo:spec} form a  \emph{Jordan Frame} for $x$. The 
$\lambda _1, \ldots, \lambda _r$ are the eigenvalues of $x$.
We 
remark that $r$ only depends on the algebra $\jAlg$. 
Given $x \in \jAlg$, we define its trace by
$$
\tr (x) = \lambda _1 + \ldots + \lambda _r,
$$
where $\lambda _1, \ldots, \lambda _r$ are the eigenvalues of 
$x$. As in the case of matrices, it turns out that the trace function 
is  linear, see Proposition~II.4.3 in \cite{FK94}.

For an element $x \in \jAlg$, we define the \emph{rank} of $x$ as the number of nonzero $\lambda _i$ that appear 
in the Equation \eqref{eq:dec}. Then, the rank of $\stdCone$ is 
$$
\matRank \stdCone = \max \{ \matRank x \mid x \in \stdCone \} = \tr (e).
$$


For the next theorem, we need the following notation. Given $x\in \jAlg$ and $a \in \Re$, we write 
$$
V(x,a) = \{z \in \jAlg \mid \jProd{x}{z} = az \}.
$$

\begin{theorem}[Peirce decomposition, see Proposition IV.1.1 in \cite{FK94}]\label{theo:peirce1}
	Let $c \in \jAlg$ be an idempotent. Then $\jAlg$ decomposes itself as an orthogonal direct sum as follows.
	$$
	\jAlg = V(c,1) \bigoplus V\left(c,\frac{1}{2}\right) \bigoplus V(c,0).
	$$	
	In addition, $V(c,1)$ and $V(c,0)$ are  Euclidean Jordan Algebras satisfying 
	$\jProd{V(c,1)}{V(c,0)} = \{0\}$.
\end{theorem}
The Peirce decomposition can be interpreted as follows. Given an element $x \in \jAlg$, 
we have the linear transformation $L_x:\jAlg \to \jAlg$ given by
$$
L_x(y) = \jProd{x}{y}.
$$
The fact that the algebra is Euclidean implies that $L_x$ is self-adjoint, therefore 
$\jAlg$ decomposes as a direct sum of the eigenspaces. Moreover, if $x$ is an idempotent, 
it can be shown that the only possible eigenvalues are $0,1,\frac{1}{2}$. So the 
decomposition in Theorem~\ref{theo:peirce1} is unsurprising. 
The remarkable part is the statement that $V(c,1)$ and $V(c,0)$ are also 
algebras and that they are  orthogonal with respect to the Jordan product as well.
%
%

%
%

\begin{lemma}\label{lemma:cone}
	Let $(\jAlg, \jProd{}{} )$ be an Euclidean Jordan Algebra and let $\stdCone$ be the cone of squares
	$$
	\stdCone = \{\jProd{z}{z} \mid z \in \jAlg\}.
	$$
	Let $x \in \jAlg$ and denote its eigenvalues by $\lambda _1, \ldots, \lambda_r$. Then $x \in \stdCone$ if and only if $\lambda _i \geq 0$ for all $i$.
\end{lemma}
\begin{proof}
	$(\Rightarrow)$ Let $c_1, \ldots, c_r$ be a Jordan Frame for $x$ such that 
	$$
	x = \sum _{i=1}^r \lambda _i c_i.
	$$
	Since the $c_i$ are idempotent, they all belong to $\stdCone$.  Therefore, 
	if the $\lambda _i$ are nonzero, it is clear that $x$ belongs to $\stdCone$ as well.
	
	$(\Leftarrow)$ Since $x \in \stdCone$, we have $x = \jProd{z}{z}$ for some $z \in \jAlg$. Take a 
	Jordan Frame for $z$:
	$$
	z = \sum _{i=1}^r \mu _i d_i .
	$$
	Item $1.$ of Theorem~\ref{theo:main} allows us to conclude that 
	$$
	x =  \sum _{i=1}^r \mu _i^2 d_i.
	$$
	Then, uniqueness implies that  each $\lambda _i$ must be among the $\mu _j^2$. In particular, all 
	the $\lambda _i$ are nonnegative. 
\end{proof}

\begin{lemma}[Exercise 3 in Chapter III of \cite{FK94}]\label{lemma:ex}
	Let $x,y \in \stdCone$. Then $\jProd{x}{y} = 0$ if and only if $\inProd{x}{y} = 0$.
\end{lemma}
\begin{proof}
	$(\Rightarrow)$ We have
	$$
	0 = \inProd{e}{\jProd{x}{y}} = \inProd{\jProd{e}{x}}{y} = \inProd{x}{y},
	$$
	due to the fact that the Jordan Algebra is Euclidean.
	
	$(\Leftarrow)$ First we show an auxiliary fact. Suppose that $\inProd{x}{\jProd{c}{c}} = 0$. 
	Consider 
	the function 
	$$
	f(z) = \frac{1}{2}\inProd{x}{\jProd{z}{z}}.
	$$
	Since $x \in \stdCone$ and $\stdCone = \stdCone^*$, $f$ is nonnegative everywhere. It follows that
	$c$ is a local minimum of $f$, therefore
	$$
	\nabla f(c) = \jProd{x}{c} = 0.
	$$
	
	Now take a Jordan frame $c_1, \ldots, c_r$ for $y$. We can 
	write
	$$
	y = \sum _{i= 1}^{\matRank y } \lambda _i c_i,
	$$
	where we suppose that only the first $\matRank y $ eigenvalues are nonzero and, therefore, 
	positive. The fact that $\inProd{x}{y} = 0$, implies that $\inProd{x}{c_i} = 0$ for 
	$i = 1, \ldots, \matRank y$. 
	
	Since the $c_i$ are idempotent, we also have $\inProd{x}{\jProd{c_i}{c_i}} = 0$. By what we have 
	just shown, we have $\jProd{x}{c_i} = 0$, for  $i = 1, \ldots, \matRank y$. 
	This implies that $\jProd{x}{y} = 0$.
\end{proof}

\begin{lemma}\label{lemma:reint_member}
	Let $\stdCone$ be a closed convex cone and let $x \in \stdCone$. Then $x \not \in \reInt \stdCone$ if and only if
	$\{x\}^\perp \cap (\stdCone^* \setminus \stdCone^\perp ) \neq \emptyset$.
\end{lemma}
\begin{proof}
	Note that $x$ does not belong to $\reInt \stdCone$ if and only if $x$ and $\stdCone$ can be \emph{properly separated}, 
	see Theorem~11.3 in \cite{rockafellar}. This means that there is a hyperplane $H = \{z \mid \inProd{z}{s} = \alpha \}$ 
	such that $x$ and $\stdCone$ belong to opposite closed half-spaces and at least one of them is not entirely 
	contained in $\stdCone$. We may assume that
	$$
	\inProd{x}{s} \leq \alpha \leq \inProd{y}{s},
	$$
	for all $y \in \stdCone$. In order for the inequality to hold, we must have $s \in \stdCone^*$.  
	Furthermore, since $x \in \stdCone$ and $0 \in \stdCone$, we conclude that $\inProd{x}{s} = 0$ and $\alpha = 0$. So that 
	$x \in H$ and $H = \{s\}^\perp$. As the separation is proper we have $\stdCone \not \subset \{s\}^\perp$. Therefore, 
	$s \in \{x\}^\perp \cap (\stdCone^* \setminus \stdCone^\perp )$.
	
	Reciprocally, by definition, the existence of $s \in \{x\}^\perp \cap (\stdCone^* \setminus \stdCone^\perp ) \neq \emptyset$
	ensures that $\{s\}^\perp$ properly separates $x$ and $\stdCone$.
\end{proof}

The following lemma is well-known and can be derived from various propositions that appear in 
\cite{FK94}, such as Proposition~III.2.2. It also follows from Equation (10) in \cite{Sturm2000}, but it appears there without proof.
For the sake of self-containment, we include a short proof below.

\begin{proposition}\label{prop:int_char}
	Let $\stdCone$ be a symmetric cone of rank $r$. The following are equivalent.
	\begin{enumerate}[$(i)$]
		\item $x \in \reInt \stdCone$
		\item $x \in \stdCone$ and $\matRank x = r$.
		\item all the eigenvalues of $x$ are positive.
	\end{enumerate}
\end{proposition}
\begin{proof}
	$(i) \Rightarrow (ii) $ Since $x \in \reInt \stdCone$, $x$ clearly belongs to $\stdCone$. 
	Write the Jordan decomposition for $x$.
	$$
	x = \sum _{i = 1}^r \lambda _i c_i. 
	$$
	Note that if $\lambda _i = 0$, then $\inProd{x}{c_i} = 0$. As $\stdCone$ is self-dual and $c_i$ is 
	idempotent, we have that $c_i \in \{x\}^\perp \cap (\stdCone^* \setminus \stdCone^\perp )$, which 
	according to Lemma~\ref{lemma:reint_member}, implies that $x \not \in \reInt \stdCone$.	
	
	$(ii) \Rightarrow (iii) $
	If $x \in \stdCone$, the eigenvalues of $x$ must be nonnegative, due to Lemma~\ref{lemma:cone}. Since we are assuming that the rank of $x$ is $r$, they must  be positive.
	
	$(iii)  \Rightarrow (i)$ Take a Jordan decomposition for $x$
	$$
	x = \sum _{i = 1}^r \lambda _i c_i. 
	$$
	Then, clearly, $x \in \stdCone$. Suppose that $x \not \in \reInt \stdCone$. Then, Lemma~\ref{lemma:reint_member} tells 
	us the existence of an $s \in \stdCone$ such that $s \not \in \stdCone^\perp$ and $\inProd{s}{x} = 0$. 
	Therefore, $\inProd{s}{c_i} = 0$, for every $i$. Which implies that
	$$
	\inProd{s}{c_1 + \ldots + c_r} = \inProd{s}{e} = 0.
	$$
	According to Lemma~\ref{lemma:ex}, we have $\jProd{s}{e} = 0$, which implies $s = 0$. This is a contradiction.
\end{proof}

\subsection{The Carath\'{e}odory number of a symmetric cone}\label{sec:car_sym}
In order to compute $\kappa(\stdCone)$, we need to know the facial 
structure of $\stdCone$. The next result follows from Theorem~2 in \cite{FB06}, due to Faybusovich. For the sake of 
self-containment we also give a proof here. In what follows, recall that $\minFacePoint{x}{\stdCone}$ indicates the minimal face of $\stdCone$ 
which contains $x$.
\begin{proposition}\label{prop:faces}
Let $\stdCone$ be a symmetric cone of rank $r$ and $x \in \stdCone$. 
Furthermore, let $c_1, \ldots, c_r$ be a Jordan frame for $x$, ordered in such a way that 
$$
x = \sum _{i = 1}^{\matRank x} \lambda _i c_i
$$
and $\lambda _1, \ldots, \lambda _{\matRank x}$ are positive. Then
\begin{enumerate}[$(i)$]
	\item $\minFacePoint{x}{\stdCone} = \stdCone \cap \{c_{\matRank x + 1} + \ldots + c_r\}^\perp$ and $\minFacePoint{x}{\stdCone}$ is the cone of 
	squares of $V(c_1+\ldots + c_{\matRank x},1)$,
	\item $\matRank \minFacePoint{x}{\stdCone}  = \matRank x$.
\end{enumerate}

In addition, $\minFacePoint{x}{\stdCone}$ is properly contained in $\stdCone$ if and only if 
$\matRank x < r$.
\end{proposition}
\begin{proof}
	 Let $s = \matRank x$ and
	\begin{align*}
	c & = c_1 + \ldots + c _{s} \\
	w & = c_ {s +1} + \ldots + c_{r}.
	\end{align*}
According to Theorem~\ref{theo:peirce1}, $V(c,1)$ is an Euclidean 
Jordan Algebra. Let $\tilde \stdFace$ denote the cone of squares 
of $V(c,1)$. Note that since $\stdCone$ is self-dual, $\{w\}^\perp$ is a 
supporting hyperplane of $\stdCone$. Therefore, $\stdCone \cap \{w\}^\perp$ is 
a face of $\stdCone$. Our first step is to show that $\tilde \stdFace = \stdCone \cap \{w\}^\perp$.
	
%
%

\fbox{ $\tilde \stdFace \subseteq \stdCone \cap \{w\}^\perp$
}
Let $y \in \tilde \stdFace$ and pick a Jordan Frame for $y$ by seeing it as an element of $V(c,1)$.
Then, 
$$
y = \lambda _1 d_1 + \ldots + \lambda _{s} d_{s},
$$
where $d_1 + \ldots + d_{s} = c$, since $c$ is the identity in $V(c,1)$. Moreover, due to  Lemma~\ref{lemma:cone}, 
the $\lambda _i$ are all nonnegative. Since $\jProd{c}{w} = 0$, we have $\inProd{c}{w} = 0$, by Lemma 
\ref{lemma:ex}. As each $d_i$ belongs to $\stdCone$, we also have $\inProd{d_i}{w} = 0$, which 
implies that $y \in w^\perp$. 

\fbox{ $\tilde \stdFace \supseteq \stdCone \cap \{w\}^\perp$
} Let $y \in \stdCone \cap \{w\}^\perp$, and following Theorem~\ref{theo:peirce1}, decompose  $y$ as
$$
y = y_1 + y_2 + y_3,
$$
with $y_1 \in V(c,1), y_2 \in V\left(c,\frac{1}{2}\right), y_3 \in V(c,0)$. Because $y \in \{w\}^\perp$, we 
have $\jProd{y}{w} = 0$, by Lemma \ref{lemma:ex}. Therefore,
\begin{align*}
\jProd{y}{w} &= \jProd{(y_1+y_2+y_3)}{(e-c)}\\
& = y_1 + y_2 + y_3 -y_1 - \frac{1}{2}y_2 \\
& = \frac{y_2}{2} + y_3 \\
& = 0
\end{align*}
Since $y_2$ and $y_3$ are orthogonal, we conclude that $y_2 = y_3 = 0$. So that 
$y = y_1$. Because $y \in \stdCone$, all its eigenvalues are nonnegative, due to Lemma~\ref{lemma:cone}.
We can also compute the eigenvalues of $y$, by seeing it as an element of $V(c,1)$.
Note that a Jordan Frame for $y$ in $V(c,1)$ can be extended to a Jordan frame for $y$ in $\jAlg$ by 	adding the remaining $c_{s+1}, \ldots, c_r$. Due to  uniqueness, it follows 
that the eigenvalues of $y$ in $V(c,1)$ are also nonnegative. Therefore, $y \in \tilde \stdFace$.

We then conclude that $\tilde \stdFace  = \stdCone \cap \{w\}^\perp$ and, therefore, $\tilde \stdFace $ is a face of $\stdCone$. Theorem 
\ref{theo:spec} guarantees that no element in $\stdCone \cap \{w\}^\perp$ has rank bigger than $s$. As $c \in \stdCone \cap \{w\}^\perp $, the rank of $\stdCone \cap \{w\}^\perp$ is indeed $s$. 
This 
proves item $(ii)$.
Moreover, due to item $(iii)$ of Proposition~\ref{prop:int_char}, $x \in \reInt (\stdCone \cap \{w\}^\perp)$.
Therefore, $\minFacePoint{x}{\stdCone} = \stdCone \cap \{w\}^\perp$.

Finally, note that if $s = r$, then $V(c,1) = V(e,1) = \jAlg$, so that 
$\minFacePoint{x}{\stdCone} = \stdCone$. Therefore, if $\minFacePoint{x}{\stdCone}$ is 
a proper face, then $s < r$. Conversely, if $s < r$, it is clear that $\minFacePoint{x}{\stdCone}$
must be proper, since it does not contain $e$.
\end{proof}

Note that if $\stdFace$ is an arbitrary face of $\stdCone$, then 
$\minFacePoint{x}{\stdCone} = \stdFace$, for all $x \in \reInt \stdFace$. So Proposition~\ref{prop:faces}
applies to all faces of $\stdCone$.

Before we proceed we need the following observation, which is a corollary to the Jordan decomposition.
\begin{corollary}\label{col:primitive}
	Let $c$ be a primitive idempotent, then 
	$$
	V(c,1) = \{\beta c \mid \beta \in \Re \}.
	$$	
\end{corollary}
\begin{proof}
	$V(c,1)$ is an Euclidean Jordan Algebra and, in fact, $c$ is the identity element in $V(c,1)$. 
	Let $x \in V(c,1)$ and consider a Jordan frame $d_1, \ldots, d_r$ for $x$. Because 
	$$
	d_1 + \ldots + d_r = c,
	$$ 
	it must be the case that $r = 1$, since $c$ is primitive. Therefore, $x = \beta c$.
\end{proof}

The next result was proved for simple Jordan Algebras in \cite{FK94}. Here, we give a more general statement.
\begin{corollary}\label{col:extreme}
	Let $\stdCone$ be a symmetric cone and $x \in \stdCone$ with $x \neq 0$. The following are equivalent.
	\begin{enumerate}[$(i)$]
		\item $x$ belongs to an extreme ray.
		\item $x$ has rank 1, {\it i.e.}, $x = \alpha c$ with $\alpha > 0$ and $c$ primitive idempotent.
	\end{enumerate}
\end{corollary} 
\begin{proof}
	$(i) \Rightarrow (ii)$ Consider a Jordan Frame for $x$ and write 
	\begin{equation*}		
	x = \sum _{i=1}^r \lambda _i c_i.
	\end{equation*}
	Due to Lemma~\ref{lemma:cone}, we have $\lambda _i \geq 0$ for all $i$.
	
	Let $\stdFace$ be the extreme ray of $\stdCone$ that contains $x$. Because $\stdFace$ is a face, if 
	$\lambda _i > 0$, then $c_i \in \stdFace$. Since $\stdFace$ has dimension one and the $c_i$ are orthogonal,  
	exactly one of the $\lambda _i$ is positive while all the others are zero.
	
	$(ii) \Rightarrow (i)$ Let $\stdFace = \minFacePoint{x}{\stdCone}$. Due to Proposition~\ref{prop:faces}, 
	$\stdFace$ has rank one and is the cone of squares of $V(c,1)$, where $c$ is a primitive 
	idempotent. Due to Corollary~\ref{col:primitive}, both $V(c,1)$ and $\stdFace$ are  one-dimensional.
\end{proof}

The Jordan decomposition together with Corollary~\ref{col:extreme} shows that 
given $x \in \stdCone$ we can write it as a sum of at most $\matRank x$ elements that live in extreme rays. That is, 
$$
\kappa (x) \leq \matRank x.
$$
The caveat is that the decomposition given by the Spectral Theorem requires that the elements be 
orthogonal to each other, while in the definition of $\kappa$ there is no such requirement. 

The next result shows that, in fact, $\kappa (x) = \matRank x$. 
This has been proven before by G\"uller and Tun\c{c}el \cite{GO98}, 
but the exposition given here is, perhaps, more elementary and does not rely on the classification of Euclidean Jordan Algebras neither on the theory of homogeneous cones as in \cite{TT03}. 
%
%

\begin{theorem}\label{theo:sym_car}
	Let $\stdCone$ be a symmetric cone and $x \in \stdCone$. We have
	$$
	\kappa (x) = \matRank x.
	$$
\end{theorem}
\begin{proof}
	The first observation is that we may assume that $x \in \reInt \stdCone$. If not, 
	we pass to the minimal face $\stdFace$ of $\stdCone$ containing $x$. Then, 
	$x \in \reInt \stdFace$ and $\stdFace$ is a symmetric cone inside some 
	Euclidean Jordan algebra, due to Proposition~\ref{prop:faces}. 
	
	Next, since $\stdCone$ is homogeneous, there is a bijective linear transformation 
	$T$ that maps $x$ to the identity $e$ and satisfies $T(\stdCone) = \stdCone$. Since $T$
	maps extreme rays to extreme rays, we have that $\kappa (x) = \kappa(e)$.
	
	We will now show that $\kappa(e) = \matRank(e)$.
	Suppose that 
	$$
	e = z_1 + \ldots + z_{\kappa(e)},
	$$
	where each $z_i$ is nonzero and belongs to an extreme ray. Due to Corollary~\ref{col:extreme}, we 
	may assume that 
	\begin{equation}\label{eq:jd_r:1}
	e = \alpha _1 d_1 + \ldots \alpha _{\kappa (e)}d_{\kappa(e)},
	\end{equation}
	where the $\alpha _i$ are positive and the $d_i$  are primitive idempotents. 
	
	Recall that if we have any Jordan frame, since the sum of idempotents is equal to $e$, the 
	eigenvalues of $e$ are all equal to one. 
	Applying the trace map at both sides of Equation \eqref{eq:jd_r:1}, we conclude that
	\begin{equation}\label{eq:jd_r:2}
	\matRank(e) = \alpha _1 + \ldots + \alpha _{\kappa (e)}.
	\end{equation}
	We now examine the following expression.
	$$
	(1 - \alpha _i)d_i = \jProd{(e-\alpha_id_i)}{d _i}.
	$$
	We take the inner product with $d_i$:
	\begin{align*}
	(1-\alpha _i)\inProd{d_1}{d_1}  & = \inProd{\jProd{(e-\alpha_id_i)}{d _i}}{d_i } \\
	& = \inProd{e-\alpha_id_i}{\jProd{d_i}{d_i}} \\
	& = \inProd{e-\alpha_id_i}{{d_i}} \\
	& \geq 0.
	\end{align*}
	The second equality follows from the fact that the algebra is Euclidean. 
	The last inequality stems from Equation \eqref{eq:jd_r:1}, which implies 
	that $e-\alpha_id_i \in \stdCone$. Since  $\inProd{d_i}{d_i} > 0$, we must 
	have $1 \geq \alpha_i$, for every $i$. In view of Equation \eqref{eq:jd_r:2}, 
	we obtain $\matRank(e) \leq \kappa (e)$.
	
	Since we already know that $\kappa (e) \leq \matRank(e)$, we have 	$\matRank(e) = \kappa (e)$.

\end{proof}

\subsection{The longest chain of faces of a symmetric cone}
\begin{theorem}\label{theo:sym_long}
	Let $\stdCone$ be a symmetric cone. We have
	$$
	\ell _{\stdCone} = \matRank \stdCone + 1.
	$$
\end{theorem}
\begin{proof}
	First, we construct a chain of faces that has length $\matRank \stdCone + 1$. Let 
	$e$ be the identity element and $c_1, \ldots, c_r$ a Jordan frame for $e$, with 
	$r = \matRank \stdCone$. Then, 
	from Proposition~\ref{prop:faces}, we have 
	$$
	\stdCone \supsetneq \minFacePoint{c_1 + \ldots + c_{r-1}}{\stdCone} \supsetneq \ldots \supsetneq \minFacePoint{c_1}{\stdCone} \supsetneq \{0\}.
	$$
	Note that the inclusions are indeed strict, since $c_i \in \minFacePoint{c_1 + \ldots + c_{i}}{\stdCone}$ but 
	$c_i \not \in \minFacePoint{c_1 + \ldots + c_{i-1}}{\stdCone}$. This shows that there is at least 
	one chain of length $\matRank \stdCone + 1$.
	
	Now suppose that we have an arbitrary chain of faces 
	$$
	\stdFace _1 \supsetneq \ldots \supsetneq \stdFace _\ell.
	$$
	We can select $\ell$ points such that $x_i \in \reInt \stdFace _i$ for all $i$. With that 
	choice, we have $\stdFace _i = \minFacePoint{x_i}{\stdCone}$. 	Due to Proposition~\ref{prop:faces}, 
	the only way that those inclusions can be strict is if $\matRank x_i > \matRank x_{i+1}$ for all $i$. 
	Since $r \geq \matRank x_1$, we conclude that $\ell$ can be at most $r+1$.

\end{proof}

The upshot of this section is that for symmetric cones we have
$$
\kappa (\stdCone) = \matRank \stdCone = \ell _{\stdCone} - 1,
$$
so the bound in Theorem~\ref{theo:main} is tight.
\section{Comments on three other classes of cones}\label{sec:conc}
To conclude this work, we will make a few comments about some
cones of matrices. 
Denote by $\CP{n}$ the cone of $n \times n$ \emph{completely positive matrices}. 
Recall that a symmetric matrix  $X$ is said to be \emph{completely positive} if 
 there is an $n\times r$ matrix $V$ such that 
 $X = VV^\T$ and all the entries of $V$ are nonnegative. 
 The smallest $r$ for which this decomposition is possible is 
 called the \emph{cp-rank} of $X$.
 
Due to a result by Berman (see \cite{berman2003completely} and also Theorem~4.2 in \cite{D11}), $Y$ belongs to an extreme ray of $\CP{n}$ if and only if 
$Y = xx^\T$ for some nonzero $x$ such that all its entries are nonnegative.
This means that the \emph{cp-rank} of $X$ coincides with $\kappa(X)$ computed 
with respect to $\CP{n}$.

Translating to our terminology, one of the  open problems described in \cite{BDM15}
is to find a nontrivial upper bound to $\kappa(\CP{n})$. It is known that 
$\kappa(\CP{n}) = n$, for $n\leq 4$ and that $\kappa(\CP{5}) = 6$. For $n\geq 6$,
the current best result is that 
$$
\kappa(\CP{n}) \leq \frac{n(n+1)}{2} - 4,
$$
see Section~4.2 in \cite{BDM15} for more information on those results. 
We cannot help but speculate whether computing $\ell _{\CP{n}}$ could 
help lower this bound. It seems  that this might be an unexplored route. In low dimension the 
bound might fail to be tight, but it is said that the geometry of $\CP{n}$ changes heavily when 
$n$ increases.

Now, let $\doubly{n}$ denote the cone of \emph{symmetric doubly nonnegative matrices}. A symmetric 
matrix $X$ belongs to $\doubly{n}$ if it is positive semidefinite and all its entries are 
 nonnegative. The importance of $\doubly{n}$ is that it can be used to relax problems 
 over $\CP{n}$ and, in fact, for $n \leq 4$, we have $\doubly{n} = \CP{n}$.
Unfortunately, for $\doubly{n}$, Theorem~\ref{theo:main} does not shed much light 
on $\kappa(\doubly{n})$, since it was shown in Proposition~26 of \cite{lourenco_muramatsu_tsuchiya4} that $\ell _{\doubly{n}} = \frac{n(n+1)}{2} + 1 $.
In low dimension we know that the bound is not tight, since we have $\kappa(\doubly{n}) = \kappa(\CP{n}) = n$ for $n\leq 4$.
However,  $\kappa(\doubly{n})$ seems to be unknown for 
large $n$. We have, nevertheless, the following easy lower bound. 
\begin{proposition}
For the cone of $n\times n$ doubly nonnegative matrices we have 
$$
\kappa(\doubly{n}) \geq n.
$$ 
\end{proposition}
\begin{proof}
First, note that if $\stdFace$ is a face of some cone $\stdCone$, then 
$\kappa(\stdCone) \geq \kappa(\stdFace)$. We will proceed by showing 
the existence of a face of $\doubly{n}$ whose Carath\'{e}odory number is equal 
to $n$.

Note that $\doubly{n} = \PSDcone{n} \cap \nonNegative{n}$, where 
$\nonNegative{n}$ is the cone of symmetric matrices with nonnegative 
entries. Let $T_n$ denote the cone of diagonal matrices with nonnegative 
entries. Note that $T_n$ is a face of $\nonNegative{n}$ and satisfies
$T_{n} =  \PSDcone{n} \cap T_{n}$. 
As $T_n$ is the intersection of a face of $\PSDcone{n}$ with a face 
of $\nonNegative{n}$, we  conclude
that $T_n$ is a face of $\doubly{n}$. As $T_n$ is polyhedral and $\dim{T_n} = n$, we 
obtain  $\kappa(T_n) =n $, by  item $i.$ of Theorem \ref{theo:bounds}. It 
follows that $\kappa(\doubly{n}) \geq n$.
\end{proof}

We also do not have much of an idea of what happens with $\ell_{\stdCone}$ and 
$\kappa (\stdCone)$ when $\stdCone$ is the cone of copositive matrices $\COP{n}$. Recall that 
a symmetric matrix $X$ is said to be copositive if $v^\T X v \geq 0$ for all $v$ with 
nonnegative entries.

Table~\ref{table:bounds_summary} summarizes what is known about the Carath\'{e}odory number and the size of 
the longest chain of faces for a few families of cones.

\begin{table}
\centering
\begin{tabular}{c|c|c}
$\stdCone$ & $\kappa (\stdCone)$ & $\ell _\stdCone$ \\  \hline
polyhedral cone of dimension $k$ & $k$ & $k+1$\\
symmetric cone of rank $r$ & $r$ & $r+1$\\
smooth cone & 2 & 3\\
$n\times n$ doubly nonnegative cone & $\geq n$ & $\frac{n(n+1)}{2} + 1 $ \\
$n\times n$ completely positive cone & nontrivial bounds are known \cite{BDM15} & ? \\
$n \times n$ copositive cone & ? & ? 
\end{tabular} \caption{Values of $\ell_{\stdCone}$ and $\kappa(\stdCone)$}\label{table:bounds_summary}
\end{table}

\small{
	\section*{Acknowledgements}
	The second author would like to thank Prof. Masakazu Muramatsu and Prof. Takashi Tsuchiya for valuable feedback 
	during the writing of this article. 
}

\bibliographystyle{plainurl}
\bibliography{bib}

\appendix
\section{Proof of Proposition \ref{prop:cone_base}}\label{app:prop}

\begin{proposition*}
	Let $C \subseteq \Re^n$ be a nonempty compact convex set. Let 
	$$
	\stdCone  = \{(\alpha, \alpha x) \mid \alpha \geq 0, x \in C \}.	
	$$
	Then
	\begin{enumerate}[$(i)$]
		\item $\stdCone $ is a pointed closed convex cone.
		\item Let $\stdFace$ be a face of $\stdCone $ that is not $\{0\}$, then 
		$$
		\stdFace _C = \{ x \in C \mid  (1,x) \in \stdFace \}
		$$is a face of $C$. Moreover, $\dim \stdFace _C = \dim \stdFace  - 1$. 
		\item Let $\stdFace _C$ be a face of $C$, then 
		$$
		\stdFace = \{ (\alpha, \alpha x)  \mid \alpha \geq 0, x \in \stdFace _C \}
		$$is a face of $\stdCone$. Moreover, $\dim \stdFace = \dim \stdFace _C + 1$.
		
	\end{enumerate}
\end{proposition*}
\begin{proof}
	\begin{enumerate}[$(i)$]
		\item Here we only show the closedness of $\stdCone$.
		Let $$\{(\alpha _k, \alpha _k x_k)  \}_{k = 1}^{+\infty} \subset \stdCone$$ be a sequence converging to $(\alpha ^*, z)$, with 
		$x_k \in C$ for all $k$.
		When $\alpha^* = 0$, the compactness of $C$ leads to $z = \lim_{k \to +\infty} \alpha_k x_k = 0$
		concluding that $(\alpha^*, z) \in \stdCone$.
		In the case $\alpha^* > 0$, we have $\alpha_k > 0$ for all sufficiently large $k$.
		Then, we see that $z/\alpha^* \in C$ since $\{(1,x_k)\} = \{\frac{1}{\alpha_k}(\alpha_k, \alpha_k x_k)\}$ converges to
		$\frac{1}{\alpha^*} (\alpha^*, z) = (1, z/\alpha^*)$ and $C$ is closed.
		Therefore, $(\alpha^*, z) \in \stdCone$.
		
		\item First of all, $\stdFace _C$ is a convex set, since it is a projection on $\Re^n$ of 
		the intersection between $\stdCone$ and the hyperplane 
		$$
		\{(\alpha, x) \mid \alpha = 1 \}.
		$$ 
		Now, let $x,y \in C$ be such that
		$$
		\gamma x + (1-\gamma)y \in \stdFace _C,
		$$for some $0 < \gamma < 1$. Therefore, 
		$$
		\gamma (1, x) + (1-\gamma)(1,y ) \in \stdFace.
		$$
		As $\stdFace$ is a face of $\stdCone$, we conclude that $(1,x),(1,y) \in \stdFace$ and that 
		$x,y \in \stdFace _C$.
		Hence, $\stdFace _C$ is a face of $C$.
		
		Let $s = \dim \stdFace _C$ and take an affinely independent subset $\{x_0,\ldots,x_s \}$ of $\stdFace_C$.
		Then, the implications
		\begin{align*}
		\sum_{i=0}^s \gamma_i (1,x_i) = 0
		& ~\Rightarrow~ \sum_{i=0}^s \gamma_i = 0\\
		~ \sum_{i=0}^s \gamma_i x_i =  0
		& ~\Rightarrow~ \gamma_0 = \cdots = \gamma_s = 0
		\end{align*}
		show that $\{(1,x_i)\}_{i=0}^s \subset \stdFace$ are \emph{linearly} independent.
		This means that $\dim\stdFace_C + 1 = s + 1 \leq \dim\stdFace$.
		Conversely, let $t = \dim\stdFace$ and $\{(\alpha_i, \alpha_i x_i)\}_{i=1}^t  \subset \stdFace$ be linearly independent.
		Then we have $\alpha_i \ne 0$ so that $\{x_i\}_{i=1}^t \subset \stdFace _C$ follows and
		its affine independence can be shown in a similar manner.
		This yields that $\dim \stdFace - 1 = t-1 \leq \dim\stdFace _C$ and therefore $\dim\stdFace _C = \dim\stdFace -1$ holds.
		
		\item It is straightforward to check that $\stdFace$ is a subset of $\stdCone$ that is a  convex cone.
		We will check that it is indeed a face. 
		Suppose that $(\alpha _1 ,\alpha _1 x_1), (\alpha _2 ,\alpha _2 x_2) \in \stdCone$ are such that 
		$$
		(\alpha _1 +\alpha _2 ,\alpha _1 x_1 + \alpha _2 x_2) \in \stdFace.
		$$
		Furthermore, suppose that both $\alpha _1$ and $\alpha _2$ are greater than zero.
		By definition, we have 
		$$
		\alpha _1 x_1 + \alpha _2 x_2 = (\alpha _1 + \alpha _2)z,
		$$ 
		for some $z \in \stdFace _C$. This means that 
		$$
		\frac{\alpha _1}{\alpha _1 + \alpha _2}x_1 + \frac{\alpha _2}{\alpha _1 + \alpha _2}x_2 = z,
		$$
		so that $x_1$ and $x_2$ belong to $\stdFace _C$. Therefore, both $(\alpha _1 ,\alpha _1 x_1)$ and  $(\alpha _2 ,\alpha _2 x_2)$ belong to $\stdFace$.
		Thus, $\stdFace$ is a face of $\stdCone$.
		
		Finally, notice that the face
		$$
		\{x \in C \mid (1,x) \in \stdFace\}
		$$
		coincides with $\stdFace _C$.
		Hence, from  assertion $(ii)$ we obtain $\dim\stdFace = \dim\stdFace_C + 1$.

		
	\end{enumerate}	
\end{proof}

\end{document}